\documentclass{amsart}
\usepackage{graphicx} 

\title{Ricci solitons from the perspectives of energy function}

\usepackage{orcidlink}

\author{Shubham Yadav \orcidlink{0000-0002-3906-5739}}
\address{Harish-Chandra Research Institute,
Chhatnag Road, Jhunsi, Prayagraj-211019, India} 
\address{Homi Bhabha National Institute,
Training School Complex, Anushakti Nagar, Mumbai 400094, India}
\email{yadavshubham3597@gmail.com}

\author{Hemangi Madhusudan Shah \orcidlink{0000-0002-2665-5094}} 
\address{Harish-Chandra Research Institute,
Chhatnag Road, Jhunsi, Prayagraj-211019, India} 
\address{Homi Bhabha National Institute,
Training School Complex, Anushakti Nagar, Mumbai 400094, India}
\email{hemangimshah@hri.res.in}


\date{\today}

\usepackage{amsfonts}
\usepackage{amsmath}
\usepackage{amssymb}
\usepackage{amsthm}

\usepackage{thmtools, thm-restate}

\usepackage{enumitem}

\usepackage[maxbibnames=99]{biblatex}
\usepackage{cleveref}

\usepackage{xcolor}

\newcommand{\abs}[1]{\left| #1 \right|}
\newcommand{\lie}[1]{\mathcal{L}_{#1}}
\newcommand{\derivative}{\mathrm{d}}

\newcommand{\brackets}[1]{\left( #1 \right)}

\DeclareMathOperator*{\ricci}{Ric}
\DeclareMathOperator*{\Qricci}{Q}
\DeclareMathOperator*{\scalar}{S}
\DeclareMathOperator*{\divergence}{div}
\DeclareMathOperator*{\trace}{tr}
\DeclareMathOperator*{\curvature}{R}
\DeclareMathOperator*{\hessian}{Hess}
\DeclareMathOperator*{\Vol}{Vol}

\theoremstyle{plain}
\newtheorem{proposition}{Proposition}

\theoremstyle{plain}
\newtheorem{corollary}{Corollary}

\theoremstyle{plain}
\newtheorem{theorem}{Theorem}

\theoremstyle{plain}
\newtheorem{lemma}{Lemma}

\theoremstyle{remark}
\newtheorem{remark}{Remark}

\newtheorem{example}{Example}

\begin{document}

\begin{abstract}
This article explores to what extent the geometry of gradient Ricci solitons extends to non-gradient Ricci solitons. 
The primary tool is the energy function $E$ of the soliton. 
We study consequences of various bounds on $E$. 
Under mild assumptions on the scalar curvature, we prove a weighted $L^1$-Liouville type theorem for both the usual Laplacian and the drifted Laplacian $\Delta_V$ associated to soliton vector field $V$, the former of which implies that Ricci solitons with bounded energy function have at most one nonparabolic end. 
Finally, we show that the measure $e^{-E} \derivative{\Vol_{g}}$ is finite for complete shrinking Ricci solitons, partially generalizing a result of Aaron Naber. 
As a consequence, non-gradient shrinking Ricci solitons also have finite fundamental groups, as in the gradient case.
\end{abstract}

\maketitle
\newpage

\tableofcontents

\section{Introduction}
A Ricci soliton is a quadruple $(M^n,g,V,\lambda)$, where $(M^n,g)$ is a complete Riemannian manifold, $V \in \chi (M),$ a complete vector field, and $\lambda \in \mathbb{R} $ satisfying
\begin{equation}\label{eqn: Ricci Soliton equation}
    \frac{1}{2} \lie{V} g + \ricci = \frac{\lambda}{2} g,
\end{equation}
where $\lie{V}$ is the Lie derivative with respect to $V$ and $\ricci$ is the Ricci tensor of $(M,g)$. 
The vector field $V$ is called the soliton vector field.
The soliton is said to be shrinking when $\lambda>0$, steady when $\lambda = 0$, or expanding when $\lambda<0$. 
Immediately, one notices that all Einstein manifolds can be considered as Ricci solitons by taking $V$ as zero or a Killing field.
Such solitons are said to be trivial.
If the soliton vector field is gradient, i.e. $V = \nabla f$, then the soliton is said to be gradient and the function $f$ is called the soliton function. 
Thus, the soliton function of a gradient Ricci soliton satisfies the equation
\begin{equation}
	\hessian f + \ricci = \frac{\lambda}{2} g.
\end{equation}
Following \cite{sym12020289}, we define the energy function of a Ricci soliton as
\begin{equation}
    E = \frac{1}{2} g(V,V).
\end{equation}

Ricci solitons which are not a priori gradient are called {\it generic} Ricci solitons (see \cite{MR3546774,MR3063556}).
This article primarily aims to explore the relationship between the geometry of Ricci solitons and the properties of their associated energy functions.
This was  also explored by Sharief Deshmukh and Hana Alsodais  in \cite{sym12020289}, primarily relying on the fact that \cref{eqn: Ricci Soliton equation} is equivalent to 
$$ \nabla_X V = \frac{\lambda}{2} X + \phi X  - \Qricci X, $$
where $Q$ is the Ricci tensor, and $\phi$ being the (1,1) tensor given by 
$$ \frac{1}{2} \derivative \eta(X,Y)=g(\phi X, Y), $$
with $\eta$ as the dual 1-form of the potential vector field $V$, i.e., 
$\eta(X) = g(X, V).$  
Our approach is quite different (independent of $\eta$)  and the geometric results obtained in the article 
are also of different flavour.

Throughout this article, we will denote the scalar curvature by $\scalar$ and the $(1,1)$-tensor field dual to Ricci tensor on $M$ by $\Qricci $. 
Thus, for any vector fields $X, Y$ on $M$, 
$$ g(\Qricci X, Y) = \ricci (X,Y) = g( X, \Qricci Y) . $$
The traceless Ricci tensor will be denoted by $\mathring{\Qricci}$,
$$ \mathring{\Qricci} X = \Qricci X - \frac{\scalar}{n} X . $$
The norm of any vector field or tensor field $T$ will be denoted by $\abs{T}$.
 

It is known that for a gradient Ricci soliton with soliton function $f$ (see Lemma 1.1  of \cite{MR2648937})
\begin{equation}\label{eqn: S + 2E - lambda f}
    \nabla (\scalar + \abs{\nabla f}^2 - \lambda f) = 0.
\end{equation}
We establish a weaker form of this for non-gradient Ricci solitons as:
\begin{restatable}{corollary}{scalarEnergyNongradient}
	\label{cor: S + 2E - lambda f for nongradient case}
	For any Ricci soliton  $(M^n,g,V,\lambda)$,    
    \begin{equation}\label{nongradient}
        \divergence ( \nabla S + \nabla E + \nabla_V V -\lambda V ) = 0 .
        \end{equation} 
\end{restatable}
Note that for a gradient Ricci soliton, $V=\nabla f$ and $\nabla E = \nabla_{\nabla f} \nabla f$. Thus, the above equation reduces to
$$ \Delta ( \scalar + 2E - \lambda f ) = 0.  $$

The boundedness of the energy function is a particularly restrictive condition for an expanding soliton. This is captured in the following theorem.
Note that a zero energy function implies that a Ricci soliton is trivial.
\begin{restatable}{theorem}{UpperBoundOnE}
	\label{thm: Upper bound on E}
	Let $(M^n,g,V,\lambda)$ be a Ricci soliton with $E$ as its energy function.
	\begin{enumerate}
		\item If the soliton is nontrivial and $E$ attains a maximum, 
			then the soliton is shrinking or steady.
			If further $\ricci(V,V)\leq 0$, 
			then the soliton is steady.  
			\item For a nontrivial soliton, if $E$ is bounded from above, then the soliton is steady or shrinking.
		\item If $\ricci (V,V) \leq 0$ and $\abs{\nabla E}\in\mathcal{L}^1(M)$, 
			then the soliton vector field is parallel and the soliton is trivial. 
	\end{enumerate}
\end{restatable}
In particular, this theorem partially recovers the well-known result by Perelman (see (2.4) of \cite{perelman2002entropyformularicciflow}): There are no nontrivial compact expanding or steady Ricci solitons.

    Let $\Omega$ be a compact subset of $M$. With respect to $\Omega$, an end of the manifold is an unbounded connected component of $M\setminus \Omega$. 
    If $\Omega_i$ is an exhaustion of $M$ by compact sets, then the number of ends with respect to $\Omega_i$ is monotonically nondecreasing sequence. 
    If this sequence is bounded, $M$ is said to have finitely many ends.
    An end is said to be parabolic if it does not admit a positive harmonic function $u$ that is constant on the its boundary and has a vanishing lim inf towards the infinity of the end. 
    Let $\mathcal{H}_D (M) $ denote the space of harmonic functions $u$ on $M$ with finite dirichlet energy, i.e.,
    $$ \int \abs{\nabla u}^2 \, \derivative{\Vol}_g < \infty . $$
    Li and Tam \cite{MR1158340} showed that the number of nonparabolic ends of a manifold is bounded above by the dimension of $\mathcal{H}_D (M)$.
    In \cite{MR3023848}, using this estimate, it was proven that a steady gradient Ricci soliton has at most one nonparabolic end.
    We show that their technique can be modified to solitons with bounded energy function and a scalar curvature lower bound. 
    Note that for any Ricci soliton, the scalar curvature $\scalar$ is always bounded below (see Theorem 2.14, \cite{MR4624811}). 
    In view of this, the assumption of lower bound on the scalar curvature is mild one.
    In particular, the hypothesis of the following theorem is satisfied by steady gradient Ricci solitons.
    \begin{restatable}{theorem}{noHarmonicWithFiniteEnergy}
        \label{thm: no harmonic function with finite energy}
        On a noncompact Ricci soliton $(M^n, g, V, \lambda)$ with $\scalar \geq \frac{(n-2)}{2} \lambda $,
    	any harmonic function $u$ with
        \begin{equation}\label{eqn: finite dirichlet energy}
            \int \abs{\nabla u}^2  \, \derivative {\Vol}_g <\infty \text{ and } 
             \int \abs{\nabla u}^2 \sqrt{E} \, \derivative {\Vol}_g <\infty
        \end{equation}
        is constant.
    \end{restatable}
    The theorem implies the following corollary, which can be seen as a generalisation of Corollary 4.2 of \cite{MR3023848}.
    \begin{corollary}
        Any noncompact Ricci soliton $(M^n, g, V, \lambda)$ with $\scalar \geq \frac{(n-2)}{2} \lambda $ and energy function $E$ bounded above has at most one nonparabolic end.
    \end{corollary}
    \begin{proof}
        We see that for a soliton with bounded energy function $E$, \Cref{thm: no harmonic function with finite energy} implies $\mathrm{dim} \mathcal{H}_D (M) = 1 $. 
        It then follows from \cite{MR1158340}, that $M$ has at most one nonparabolic end.
    \end{proof}

    The drifted Laplacian, with respect to a vector field $V$, of a function $u$ is defined as 
    $\Delta_{V} u =   \Delta u - g(V,\nabla u) $. 
    If the vector field $V$ is gradient, say $V=\nabla f$, the operator 
    $\Delta_f u = \Delta u - g(V, \nabla u)$ is called the weighted Laplacian. 
    The weighted Laplacian $\Delta_f$ is a self-adjoint operator with respect to the measure 
    $e^{-f} \derivative{\Vol} g$.
    A function $u$ is said to be $V$- harmonic if $\Delta_V u = 0$. Similarly, it is said to be $f$-harmonic if $\Delta_f u = 0$. 
    In \cite{MR2843238}, Munteanu and Wang studied the dimension of the space of $f$-harmonic functions with some bounds.
    Similar to the previous theorem, we have the following result regarding $V$-harmonic functions on a noncompact Ricci soliton.
    \begin{restatable}{theorem}{NoVHarmonic}
    	\label{thm: No V-harmonic function with constant energy and scalar bounded below}
        On a noncompact Ricci soliton $(M^n, g, V, \lambda)$ with $\scalar> \frac{(n-2)}{2} \lambda $,
    	any $V$-harmonic function $u$ with
        $$ \int \abs{\nabla u}^2 E \, \derivative {\Vol}_g <\infty \text{ and } 
             \int \abs{\nabla u}^2 \, \derivative {\Vol}_g <\infty $$
        is constant.
    \end{restatable}

Moreover,  we also obtain bounds on scalar curvature on Ricci solitons when scalar curvature is subharmonic with respect to the drifted Laplacian $\Delta_V$ (see \Cref{Cor: bounds on constant scalar curvature}). 
We see in \cite{MR2507581} that a constant scalar curvature is a very restrictive condition on gradient Ricci solitons.
We explore some of these results in 
Section \ref{main} for the nongradient case.
We also give examples of Cigar soliton and $\mbox{Sol}_3$ solitons  supporting  some of our results. 

Theorem 1 of \cite{naber2006geometryanalysisriccisolitons} implies that the measure $e^{-f} \derivative {\Vol}_g $ of a complete shrinking gradient Ricci soliton $(M^n,g,f,\lambda)$ with bounded Ricci curvature is finite. 
We extend this result to complete, nongradient shrinking Ricci solitons, by using similar techniques as in this paper. In fact, we show that:  
\begin{restatable}{theorem}{shrinkingRShasfiniteEVolume}
	\label{thm: bounded Ricci implies finite E volume}
	Let $(M^n,g,V,\lambda)$ be a complete shrinking Ricci soliton with bounded Ricci curvature. 
	Then the measure $e^{-E} \derivative {\Vol}_g$ is finite. Consequently,  $M$ has finite fundamental group.
\end{restatable}
The finiteness of a weighted volume for a generic shrinking Ricci soliton is a new result
whereas the finiteness of the fundamental group of a complete shriking Ricci soliton was proven by Wylie in \cite{MR2373611} without any bound on Ricci curvature through a different approach. 
As a particular case of the above theorem, we partially recover the aforementioned result of \cite{naber2006geometryanalysisriccisolitons} for gradient Ricci solitons. 
See \Cref{cor: shrinking GRS has finite f volume} for details.

The paper is divided into 6 sections. 
The Section \ref{prelim}  is devoted to Preliminaries, Section \ref{der} explores  general computations on  Ricci solitons, while the Section \ref{main} investigates interplay between the scalar curvature, energy function and soliton vector field while the Section \ref{liou} investigates the $L^1$-Liouville type theorems on Ricci solitons.
The final section namely, \Cref{finite-volume} generalizes a particular result of Aaron Naber \cite{naber2006geometryanalysisriccisolitons} to nongradient shrinking Ricci solitons.

\section{Preliminaries}\label{prelim}
In  this section,  we describe some of the basics  on any Riemannian manifold $(M^n,g)$  pertaining to Ricci solitons. The concepts and basic formulae  can be found  in \cite{MR3469435}  and  \cite{MR284950}.

In the sequel, the following Bochner's formula will be used. 
\begin{equation}\label{eqn: Bochner formula}
    \Delta\left( \frac{1}{2}\left| X \right|^2 \right) = g( \Delta X, X) + \left| \nabla X \right|^2 ,
\end{equation}
where $\Delta$ on the right hand side indicates the connection Laplacian and is defined as 
\begin{equation*}
    \Delta X =  \trace ( \nabla^2 X ). 
\end{equation*}
The gradient version of the Bochner formula is 
\begin{equation}\label{eqn: Bochner formula for gradient}
        \Delta \brackets{\frac{1}{2} \abs{\nabla u}^2} = \ricci \brackets{\nabla u, \nabla u} 
            + g \brackets{ \nabla u , \nabla \brackets{\Delta u} } + \abs{ \hessian u}^2. 
\end{equation}

The Lie derivative of the connection with respect to a vector field $V$  on $M$ is defined as follows:
\begin{equation*}
    (\lie{V}\nabla)(X,Y) = \lie{V}(\nabla_X Y) - \nabla_{\lie{V}X} Y - \nabla_X(\lie{V}Y) .
\end{equation*}
We define the curvature tensor as 
$$ \curvature (X,Y) Z = \nabla^2_{X,Y} Z - \nabla^2_{Y,X} Z 
	= \nabla_X \nabla_Y Z - \nabla_Y \nabla_X Z - \nabla_{[X,Y]} Z , $$
and its Lie derivative with respect to $V$ as
$$ (\lie{V} \curvature) (X,Y)Z = \lie{V} ( \curvature(X,Y)Z ) 
	- \curvature (\lie{V} X,Y)Z -\curvature (X, \lie{V}Y)Z - \curvature (X, Y) \lie{V}Z.  $$

The following formulas can be found in \cite{MR284950}, we present their proofs here for the sake of convenience.
\begin{proposition}[Yano's formulae]\label{Yano}
 For any vector fields $X, Y, Z$ on  a Riemannian manifold $(M^n,g)$ we have, 
 \begin{gather}
        \label{eqn: lie V of nabla}
            \left( \lie{V} \nabla \right) (X,Y) = \nabla^2_{X,Y} V + \curvature(V,X)Y, \\
        \label{eqn: Nabla X of L V g}
            \left( \nabla_X \left( \lie{V}g \right) \right) (Y,Z) 
                = g( \nabla^2_{X,Y} V, Z ) + g( Y, \nabla^2_{X,Z} V ),\\
        \label{eqn: Yano commutation formula}
            \left( \nabla_X \left( \lie{V}g \right) \right) (Y,Z) = 
                g \left( \left( \lie{V} \nabla \right) (X,Y) , Z \right) + 
                g \left( \left( \lie{V} \nabla \right) (X,Z) , Y \right), \\
        \label{eqn: Lie V of curvature}
            (\lie{V}\curvature)(X,Y) Z = \left( \nabla_X \left( \lie{V}\nabla \right) \right) (Y,Z) -
        \left( \nabla_Y \left( \lie{V}\nabla \right) \right) (X,Z) .
    \end{gather}
\end{proposition}
\begin{proof}
    The first two equations follow from unwinding the definition.
    \begin{align*}
        \left( \lie{V} \nabla \right) (X,Y) 
            &= \nabla_V \nabla_X Y - \nabla_{\nabla_X Y} V 
                - \nabla_{\nabla_V X - \nabla_X V} Y - \nabla_X (\nabla_V Y - \nabla_Y V), \\
            &= \nabla^2_{V,X} Y - \nabla^2_{X,V} Y + \nabla^2_{X,Y}V = \nabla^2_{X,Y} V + \curvature(V,X)Y .
    \end{align*}
    \begin{align*}
        \left( \nabla_X \left( \lie{V}g \right) \right) (Y,Z)
            &= X \brackets{ \lie{V}g (Y,Z) } - (\lie{V}g)(\nabla_X Y,Z) - (\lie{V}g)(Y, \nabla_X Z), \\
            &= X \brackets{ g(\nabla_Y V, Z) + g(Y, \nabla_Z V) } 
            \\ &\qquad - (\lie{V}g)(\nabla_X Y,Z) - (\lie{V}g)(Y, \nabla_X Z), \\
            &= g( \nabla^2_{X,Y} V ,Z ) + g( Y, \nabla^2_{X,Z} V ).
    \end{align*}
    The third equation follows from the first two and the symmetries of the curvature tensor.
    \begin{align*}
        g ( ( \lie{V} \nabla ) (X,Y) &, Z ) 
            + g \left( \left( \lie{V} \nabla \right) (X,Z) , Y \right) \\ 
                &= g(\nabla^2_{X,Y} V,Z) + g(\nabla^2_{X,Z} V,Y) 
                		+ \curvature(V,X,Y,Z) +\curvature(V,X,Z,Y), \\
                &= \left( \nabla_X \left( \lie{V}g \right) \right) (Y,Z).
    \end{align*}
    The last equation follows from \cite[eqn 5.16]{MR284950} or can be computed directly. 
\end{proof}

It is well known that for any function $f$ and vector field $X$, 
$$\divergence \hessian f (X) = \ricci (\nabla f, X) + X(\Delta f). $$
The following is the nongradient version of the above formula (see also Lemma 2.1 of \cite{MR2507581}).

\begin{proposition}\label{prop: Divergence of L V g}
 Any two vector fields $V, X$ on a Riemannian manifold $(M^n,g)$  
satisfy:
\begin{equation}\label{eqn: Divergence of Lie derivative}
        [ \divergence ( \lie{V}g ) ] X = g (\Delta V , X) + \ricci(V,X) + X (\divergence V).
    \end{equation}
\end{proposition}
\begin{proof}
For a point $p\in M$, let $\{E_i\}$ be an orthonormal frame in a neighbourhood of 
point $p$ which is parallel at $p$.  At point $p \in M$ we get,
    \begin{equation*}
        X( \divergence V ) = X \left( \sum_{i=1}^n g(E_i, \nabla_{E_i} V) \right),
    \end{equation*}
    which implies
    \begin{equation}\label{eqn: derivative of divergence}
        X( \divergence V ) = \sum_{i=1}^n g( E_i, \nabla^2_{X,E_i} V ).
    \end{equation}
    Now, using \cref{eqn: Nabla X of L V g} at $p$ we obtain,
    \begin{align*}
        [ \divergence ( \lie{V}g ) ] X &= \left(\nabla_{E_i} ( \lie{V}g ) \right) (E_i, X), \\
            &= g( \nabla^2_{E_i,E_i} V , X ) + g( E_i, \nabla^2_{E_i,X} V ), \\
            &= g(\Delta V, X) + g( E_i, \curvature(E_i, X)V ) + g(E_i, \nabla^2_{X,E_i}V), \\
            &= g(\Delta V, X) + \ricci(X,V) + X(\divergence V),
    \end{align*}
    where the last equality uses \cref{eqn: derivative of divergence}.
\end{proof}

\vspace{0.1in}


\section{Basic Identities on Ricci solitons}\label{der}
This section focuses on extending Proposition $2.1$ of \cite{MR2448435} to the nongradient case.

\begin{proposition}
    For a Ricci soliton $(M^n,g,V,\lambda)$  the following holds:
    \begin{gather}
        \divergence V = \frac{n\lambda}{2} - \scalar \label{eqn: Divergence of V} , \\
        -2 \Delta V = 2 \Qricci V =  
            \lambda V - \nabla E - \nabla_V V \label{eqn: Q of V} , \\
        \abs{\nabla V }^2 - \Delta E = \ricci(V,V) = \lambda E - V E \label{eqn: ricci V V}.
    \end{gather}
\end{proposition}
\begin{proof}
    \Cref{eqn: Divergence of V} is obtained by contracting the Ricci soliton equation (\cref{eqn: Ricci Soliton equation}). 
    Taking divergence of \cref{eqn: Ricci Soliton equation}, using second contracted Bianchi identity, and comparing with \cref{eqn: Divergence of Lie derivative}, one obtains
    \begin{equation*}
        \Delta V  + \Qricci V  + \nabla (\divergence V) = - \nabla\scalar.
    \end{equation*}
    Substituting \cref{eqn: Divergence of V} above gives left half of \cref{eqn: Q of V}. 
    From \cref{eqn: Ricci Soliton equation}, we obtain
    \begin{equation*}
    \ricci (X,V) = \frac{\lambda}{2} g(X,V) - \frac{1}{2} (\lie{V}g)(X,V) 
        = \frac{\lambda}{2} g(X,V) - \left. \frac{1}{2} \middle( X E + g(X,\nabla_V V) \right) .
    \end{equation*}
    Thus, 
    \begin{equation*}
        2 \Qricci V =  \lambda V - \nabla E - \nabla_V V ,
    \end{equation*}
    which proves the right half of \cref{eqn: Q of V}. Taking inner product of \cref{eqn: Q of V} with $V$ and using \cref{eqn: Bochner formula} gives \cref{eqn: ricci V V}.
\end{proof}

\begin{lemma}\label{lem: Bochner formula w.r.t. Delta V}
    On a Ricci soliton $(M^n,g,V,\lambda)$, any smooth function $u$ satisfies the following Bochner formula
    \begin{equation}\label{eqn: Bochner formula w.r.t. Delta V}
        \Delta_V \brackets{\frac{1}{2} \abs{\nabla u}^2} = \frac{\lambda}{2} \abs{\nabla u}^2 
            + g\brackets{ \nabla u , \nabla ( \Delta_V u) } + \abs{\hessian u}^2 .
    \end{equation}
\end{lemma}
\begin{proof}
    By definition, $\Delta u = \Delta_V u + g(V, \nabla u) $.
    Thus,
    \begin{align*}
        g\brackets{ \nabla u , \nabla ( \Delta u) } 
            &= g\brackets{ \nabla u , \nabla ( \Delta_V u) }  
                + g\brackets{ \nabla_{\nabla u} V , \nabla u } + g (V, \nabla_{\nabla u} \nabla u) , \\
            &= g\brackets{ \nabla u , \nabla ( \Delta_V u) } 
                + \frac{1}{2} \lie{V}g (\nabla u, \nabla u) + \hessian u (V, \nabla u).
    \end{align*}
    Substituting in the Bochner formula (\cref{eqn: Bochner formula for gradient})
    and using the Ricci soliton equation, we  obtain
    \begin{equation*}
        \Delta \brackets{\frac{1}{2} \abs{\nabla u}^2} 
            = \lambda \brackets{ \frac{1}{2} \abs{\nabla u}^2 } + g\brackets{ \nabla u , \nabla ( \Delta_V u) } 
                + \hessian u (V, \nabla u) + \abs{ \hessian u}^2 .
    \end{equation*}
    Laslty, noting 
    $$ V \brackets{\frac{1}{2} \abs{\nabla u}^2} = g\brackets{ \nabla_V \nabla u,\nabla u } = \hessian u (V, \nabla u)  $$
    gives the desired equation.
\end{proof}

The following is a known result (see Lemma 2.1 of \cite{MR3063556}). It allows for an estimate of $V$-Laplacian of the scalar curvature. We share a different proof based on Yano's formulae derived in Proposition \ref{Yano}
\begin{restatable}{theorem}{LaplacianScalarRicci}
	\label{prop: laplacian Scalar in norm ricci}
    For a Ricci soliton $(M^n,g,V,\lambda)$, the following holds:
    \begin{gather}
        \Delta\scalar - V\scalar - \lambda\scalar = - 2\abs{\Qricci}^2,   \label{eqn: laplacian Scalar in norm ricci} \\
        \Delta\scalar - V\scalar - \frac{2\scalar}{n} \left( \divergence V \right) 
            = -2 \abs{\mathring{\Qricci}}^2 .\label{eqn: laplacian Scalar in traceless ricci}
    \end{gather}
\end{restatable}


\begin{proof}
The second equation follows from the first by noting $\abs{\mathring{\Qricci}}^2 = \abs{\Qricci}^2 - \frac{\scalar^2}{n} $ and \cref{eqn: Divergence of V}.

Differentiating \cref{eqn: Ricci Soliton equation} along an arbitrary vector field, we obtain 
\begin{equation*}
    \nabla_X \left( \lie{V} g \right) = - 2 \nabla_X \ricci.
\end{equation*}

Taking cyclic sum of \cref{eqn: Yano commutation formula} (with subtraction on the last one) and using the above gives
\begin{align*}
     2 & g\left( \left( \lie{V}\nabla  \right)  (X,Y), Z \right) \\
        &= \left( \nabla_X \left( \lie{V} g \right) \right) (Y,Z) + \left( \nabla_Y \left( \lie{V} g \right) \right) (Z,X)
            - \left( \nabla_Z \left( \lie{V} g \right) \right) (X,Y) \\
        &= -2 \left( \left( \nabla_X \ricci \right) (Y,Z) + \left( \nabla_Y \ricci \right) (Z,X) - \left( \nabla_Z \ricci \right) (X,Y) \right).
\end{align*}

For simplicity of computation we define a $(1,1)$ tensor $T$ by
\begin{equation*}
    g ( T(X,Y) , Z ) = (\nabla_Z \ricci ) (X,Y).
\end{equation*}
Also for any $(1,2)$ tensor $\mathcal{U}$, we denote its symmetrization by
\begin{equation*}  
 (\mathcal{U} \circ \mathrm{sym}  ) (X,Y) = \mathcal{U}(X,Y) + \mathcal{U}(Y,X).
\end{equation*} 
Then we can write
\begin{equation*}
	\left( \lie{V}\nabla \right) (X,Y) = - \left( (\nabla_X \Qricci ) Y + (\nabla_Y \Qricci) X - T (X,Y) \right).
\end{equation*}
Or simply
\begin{equation*}
	\lie{V}\nabla = - \left( \nabla\Qricci \circ \mathrm{sym} - T \right).
\end{equation*}
Therefore,
\begin{equation*}
    \nabla_X \left( \lie{V}\nabla \right) 
    		= - \left( \nabla_X \nabla\Qricci \circ \mathrm{sym} - \nabla_X T \right) .
\end{equation*}
Now using the above in \cref{eqn: Lie V of curvature},
\begin{multline*}
    (\lie{V}\curvature) (X,Y) Z = 
        - \left( \left( \nabla^2_{X,Y}\Qricci \right) Z 
        + \left( \nabla^2_{X,Z}\Qricci \right) Y - \nabla_X T (Y,Z)  \right) \\
        + \left( \left( \nabla^2_{Y,X}\Qricci \right) Z 
        + \left( \nabla^2_{Y,Z}\Qricci \right) X - \nabla_Y T (X,Z)   \right). 
\end{multline*}

Let $\{E_i\}$ be an orthonormal frame in the neighbourhood of a point point $p\in M$ which is parallel at $p$. 
Noting,
\begin{equation*}
    g( (\nabla_X T)(Y,Z) , W ) = (\nabla^2_{X,W} \ricci) (Y,Z),
\end{equation*}
and contracting the above equation w.r.t. $X$, we obtain
\begin{multline}\label{eqn: Lie V of Ricci}
	(\lie{V}\ricci)(Y,Z) = \left( -\nabla^2_{E_i, Y}  \ricci \right) (Z, E_i) - \left( \nabla^2_{E_i, Z} \ricci \right) (Y , E_i) + \left( \nabla^2_{E_i, E_i}  \ricci \right) ( Y, Z)  \\
		+ \left( \nabla^2_{Y, E_i} \ricci\right)  (Z, E_i) + \left( \nabla^2_{Y, Z} \ricci \right) (E_i, E_i) - \left( \nabla^2_{Y, E_i} \ricci \right) (E_i, Z)  .
\end{multline}

The contracted Bianchi identity implies,
\begin{gather*}
    (\nabla_{E_i} \ricci) (E_i , Z ) = \frac{1}{2} \mathrm{d} \scalar (Z), \\
    (\nabla_Y \nabla_{E_i} \ricci) (E_i, Z) + (\nabla_{E_i} \ricci) (E_i, \nabla_Y Z)
        = \frac{1}{2} ( \nabla_Y \mathrm{d}\scalar )(Z) + \frac{1}{2} \mathrm{d}\scalar (\nabla_Y Z), \\
    (\nabla^2_{Y,E_i} \ricci) (E_i, Z) = \frac{1}{2} \hessian \scalar (Y,Z).
\end{gather*}
Furthermore,
\begin{align*}
    (\nabla^2_{E_i,Y} \ricci) (E_i,Z) &= (\curvature(E_i, Y)\ricci)(E_i, Z) + (\nabla^2_{Y,E_i} \ricci) (E_i, Z), \\
        &= (\curvature(E_i, Y)\ricci)(E_i, Z) + \frac{1}{2}\hessian\scalar (Y, Z).
\end{align*}
Substituting in \cref{eqn: Lie V of Ricci},
\begin{multline}\label{ric}
    (\lie{V}\ricci)(Y,Z) = 
            -(\curvature(E_i,Y)\ricci )(E_i,Z) - \frac{1}{2}\hessian\scalar(Y,Z)  \\
            -(\curvature(E_i,Z)\ricci )(E_i,Y) - \frac{1}{2}\hessian\scalar(Z,Y) \\
            + ({\Delta} \ricci)(Y,Z) + \frac{1}{2} \hessian\scalar(Y,Z) + \hessian\scalar(Y,Z) - \frac{1}{2} \hessian\scalar(Y,Z).
\end{multline}
Noting
\begin{align} 
    (\curvature(E_i,Y)\ricci )(E_i,Z) &= - \ricci ( \curvature(E_i,Y) E_i, Z ) - \ricci ( E_i, \curvature(E_i,Y)Z ), \\\label{Q}
        &= g( \Qricci Y, \Qricci Z ) - \curvature( E_i, Y, Z, \Qricci E_i ). 
\end{align}
 Using eq.(\ref{Q}) in eq. (\ref{ric})
 and simplifying we obtain,
\begin{multline}
    (\lie{V}\ricci)(Y,Z) = 
            -g( \Qricci Y, \Qricci Z ) + \curvature(E_i, Y,Z, \Qricci E_i) \\
            -g( \Qricci Z, \Qricci Y ) + \curvature(E_i, Z,Y, \Qricci E_i)
            +({\Delta} \ricci)(Y,Z) .
\end{multline}
Tracing the above equation,
\begin{equation*}
    \lie{V}g (\Qricci E_i, E_i) + 2 V \scalar = \Delta \scalar.
\end{equation*}
Now \cref{eqn: Ricci Soliton equation} gives,
\begin{equation*}
    2(\lie{V}g)(E_i, \Qricci E_i) = 2\lambda \scalar -4 \abs{\Qricci}^2.
\end{equation*}
Therefore,
\begin{equation*}
    \lambda\scalar - 2\abs{\Qricci}^2 + V\scalar = \Delta\scalar.
\end{equation*}

\end{proof}

The rest of this section is dedicated to proving \Cref{cor: S + 2E - lambda f for nongradient case}.
It will involve various auxiliary identities involving the scalar curvature and the energy function.
We begin by establishing the following lemma.

\begin{lemma}\label{lem: divergence of TX}
Let $(M^n,g)$ be a  Riemannian manifold. Then  for  any symmetric $(1,1)$-tensor $T$ on $M$, and any vector field $X$ on $M$,
    $$ \divergence (TX) = (\divergence T) X + \frac{1}{2}g\left(\mathring{T},\lie{X}g\right) + \frac{\trace T}{n} \divergence X .$$
\end{lemma}
\begin{proof}
	Let $(E_i)$ be an orthonormal frame in some neighbourhood of a point $p$.
	Then  
	\begin{align*}
		g (T , \lie{X} g ) &= g (T (E_i) , E_j)~ (\lie{X} g )(E_i, E_j) \\
		& = T_{i}^j \left( g ( \nabla_{E_i} X , E_j) + g ( E_i, \nabla_{E_j} X \right)) \\	
		&  = 2 g ( T E_i , \nabla_{E_i} X ) =  2 g ( E_i , T( \nabla_{E_i} X ) ) .
	\end{align*}
	Now  for any vector field $Y$, 
	$$ \nabla_Y ( T X ) = (\nabla_Y T) X + T (\nabla_Y X).  $$
	Contracting the above equation over $Y$,
	$$ g(\nabla_{E_i}(TX) , E_i) = g \left( (\nabla_{E_i} T ) X , E_i) + g ( T (\nabla_{E_i} X), E_i )\right). $$
	Hence,
	\begin{equation*}
		\divergence (TX) = (\divergence T ) X + \frac{1}{2} g ( T , \lie{X} g ).
	\end{equation*}
	Substituting $ \mathring{T} = T - \frac{\trace T }{n} g$ produces the result.
\end{proof}

\begin{remark}
For a closed manifold, integrating the above equation  implies  Theorem 4.2 of \cite{MR4393949}.  
\end{remark}

\begin{proposition}\label{lapE}
    For any Ricci soliton  $(M^n,g, V, \lambda)$,      
   \begin{equation}\label{eqn: V scalar norm of traceless Ricci} 
        \left. \frac{1}{2} \middle( V\scalar + \divergence (\nabla E) + \divergence (\nabla_V V)  \right) 
            = \frac{1}{n} (\divergence V)^2 + \abs{\mathring{Q}}^2 .
    \end{equation}
\end{proposition}
\begin{proof}
	Taking $T=\Qricci$ and $X=V$ in \Cref{lem: divergence of TX}, we have
    $$ \divergence (\Qricci V) = (\divergence \Qricci) V + \frac{1}{2}g\left(\mathring{\Qricci},\lie{V}g\right) + \frac{\trace T}{n} \divergence V . $$
    Now availing \cref{eqn: Ricci Soliton equation} and second contracted Bianchi identity viz., 
  $ 2  \divergence \ricci = \derivative \scalar $, we obtain
    $$ \divergence ( \Qricci V ) = \frac{1}{2} V\scalar - \abs{\mathring{Q}}^2 + \frac{\scalar}{n} \divergence V .  $$
    Comparing the above with \cref{eqn: Q of V},
    \begin{equation*}
        \frac{1}{2} V\scalar - \abs{\mathring{Q}}^2 + \frac{\scalar}{n} \divergence V 
        		= \divergence (\Qricci V) 
            = \left. \frac{1}{2} \middle( \lambda \divergence V   - \Delta E - \divergence (\nabla_V V)\right). 
    \end{equation*}
    Or simply
    \begin{equation}\label{div E}
        \left. \frac{1}{2} \middle( V\scalar + \divergence (\nabla E) + \divergence (\nabla_V V)  \right) 
            = \abs{\mathring{Q}}^2 - \frac{\scalar}{n} (\divergence V) + \frac{\lambda}{2} \divergence V .
    \end{equation}
    Lastly \cref{eqn: Divergence of V} implies
    \begin{equation*}
        \left. \frac{1}{2} \middle( V\scalar + \divergence (\nabla E) + \divergence (\nabla_V V)  \right) 
            = \frac{1}{n} (\divergence V)^2 + \abs{\mathring{Q}}^2 .
    \end{equation*}
\end{proof}

\begin{remark}
    For any noncompact Ricci soliton with constant scalar curvature, the above corollary implies the existence of a vector field $W = \nabla E + \nabla_V V $ which has a nonnegative divergence, and $W$ is divergence-free iff the  soliton is trivial (Einstein).
\end{remark}

\scalarEnergyNongradient*
\begin{proof}
From \cref{eqn: laplacian Scalar in traceless ricci} and \cref{div E}
\begin{equation*}
	\Delta \scalar 
		= V\scalar -2 \abs{\mathring{\Qricci}}^2 + \frac{2\scalar}{n} (\divergence V) 
		= \lambda \divergence V -\divergence (\nabla E + \nabla_V V).
\end{equation*}

\end{proof}
Thus we  affirm  \Cref{cor: S + 2E - lambda f for nongradient case} as stated in the introduction.


\section{Interplay between the scalar curvature, energy function and soliton vector field}\label{main}
In this section, we establish 
\Cref{thm: Upper bound on E} as stated in the introduction.   
To prove \Cref{thm: Upper bound on E}, we will use the Omori-Yau maximum principle,
\Cref{thm:Cheng Yau maximum principle}, and 
\Cref{lem: Subharmonic with finite energy is harmonic}.
In the later part of the section, motivated by the results of \cite{MR2507581}, we explore Ricci solitons of constant scalar curvature. In particular, we obtain  
bounds on scalar curvature and the energy function. 

First we wish to showcase a result from \cite{MR3721584} which shows the importance of the energy function. 
Rewriting \cref{eqn: S + 2E - lambda f} as 
$$ \lambda \nabla f = \nabla ( \scalar + 2E ), $$
we see that for a given nonsteady Ricci soliton $(M^n,g,V,\lambda)$, up to a constant, the best candidate for a soliton function is $\frac{1}{\lambda} (\scalar + 2E)$.
Sharma et. al. showed that this is indeed the case when $\nabla V$ is symmetric. 

\begin{restatable}{theorem}{SymmetricImpliesGradient}
	\label{thm: nabla V symmetric implies gradient}
    For a non-steady Ricci soliton $(M^n,g,V,\lambda)$, if $\nabla V$ is symmetric, then the soliton is gradient.
\end{restatable}

\begin{proof}
    When $\nabla V$ is symmetric, we have
    \begin{gather*}
        \lie{V}g (X,Y) = 2g(\nabla_X V, Y) = 2g(X, \nabla_Y V) \text{ and } \\
        (\nabla_Z (\lie{V}g))(X,Y) = 2g(\nabla^2_{Z,X} V, Y) = 2g(X, \nabla^2_{Z,Y} V) .
    \end{gather*}
    It then follows that (also see proof of \Cref{prop: Divergence of L V g}),
    $$ (\divergence (\lie{V} g) ) Y = 2g (\Delta V, Y) = 2\ricci(V,Y) + 2 Y(\divergence V) . $$
    Therefore, using \cref{eqn: Ricci Soliton equation} and the second contracted Bianchi identity, one obtains
    \begin{equation*}
        -\nabla\scalar = -2 \divergence\Qricci = \divergence (\lie{V} g) = 2 \Delta V = 2 ( \Qricci V + \nabla(\divergence V) ) .
    \end{equation*}
    Using \cref{eqn: Divergence of V} above,
    \begin{equation}\label{eqn: Q V when nabla V is symetric}
        \Qricci V = \frac{1}{2} \nabla\scalar = - \Delta V.
    \end{equation}
    Now, the definition of $E$ and the symmetry of $\nabla V$ implies
    $$ X E = g(\nabla_X V, V) =  \frac{1}{2} \lie{V} g (X,V) = \frac{\lambda }{2} g(X,V) - \ricci (X,V), $$
    and thus,
    $$ \nabla E = \frac{\lambda}{2} V - \Qricci V .$$
    Comparing with \cref{eqn: Q V when nabla V is symetric},
    $$ \lambda V = \nabla \brackets{ 2 E + S }  . $$
\end{proof}

Now we  give the different and shorter proof of the following lemma  than \cite[p. 660]{MR4518934}.
In what follows we will work with
the cut off function $\phi$ on $M$ which is supported 
on $B(p, 2R)$ and satisfying $\abs{\nabla \phi} \leq \frac{C}{R}$ on $B(p, 2R)$.

\begin{lemma}\label{lem: Subharmonic with finite energy is harmonic}
    For a subharmonic function $f$ on a noncompact manifold $(M^n,g)$, if $\abs{\nabla f} \in \mathcal{L}^1 (M) $, then $\Delta f = 0$.
\end{lemma}
\begin{proof}
    Fix point $p$, and let $A_R = B(p,2R)\setminus B(p,R)$.  We consider 
   $\phi$ to be the cutoff function supported in $B(p,2R)$,  such that 
   $\abs{\nabla \phi}  \leq \frac{C}{R}$.  Then,
    \begin{multline*}
        \int_M \phi \Delta f \, \derivative {\Vol}_g
            = - \int_{A_R} g\brackets{\nabla \phi, \nabla f} \, \derivative {\Vol}{_g}
            \leq  \int_{A_R} \abs{\nabla \phi} \abs{\nabla f} \, \derivative {\Vol}{_g} \\
            \leq \frac{C}{R} \int_{A_R}  \abs{\nabla f} \, \derivative {\Vol}{_g} \to 0 \text{ as } R\to\infty.
    \end{multline*}
    Since the integrand on the LHS is positive, the result follows from the monotone convergence theorem.
    \end{proof}

In \cite{MR3546774}, it was proven that for any (generic) Ricci soliton with soliton vector field $V$, the Omori-Yau maximum principle holds for the operator $\Delta_V$.
Effectively, we have
\begin{theorem}[\cite{MR3546774}]\label{thm:Cheng Yau maximum principle}
Let $(M^n, g,V,\lambda)$ be a Ricci soliton.
Then for any $u\in C^{2}(M)$ with $u^*=\sup_{M}u<\infty$,  there exists a sequence of points $x_{k}\in M$ such that
\begin{equation*}
    \lim_{k\to \infty} u(x_k) > u^*-\frac{1}{k}, \quad \abs{\nabla u} (x_k) < \frac{1}{k}, \quad \Delta_V u (x_k) < \frac{1}{k}.
\end{equation*}
\end{theorem}

\vspace{0.1in}

\UpperBoundOnE*
\begin{proof}
For a nontrivial Ricci soliton, the energy function $E$ is not identically zero.
	\begin{enumerate}
		\item \Cref{eqn: ricci V V} implies
	    		\begin{equation}\label{V}
    		    		\Delta E + \lambda E = \abs{\nabla V}^2 + g(V, \nabla E) .
    			\end{equation}
    			If $E$ attains a maximum at $p$, then $\nabla E |_p= 0$ and $\Delta E |_p \leq 0 $. Thus, $\lambda \geq 0$. 
    		
    			If further, $\ricci (V,V) \leq 0$, the other part of \cref{eqn: ricci V V} implies
    			$$ \Delta E = \abs{\nabla V}^2 - \ricci (V,V) \geq 0. $$
    			Thus $E$ is subharmonic and if it attains a maximum, it must be constant. 
    			Then
 		   	$\abs{\nabla V} = 0$ and $\lambda E =0$, which gives the result.
    		\item The Omori-Yau maximum principle (\Cref{thm:Cheng Yau maximum principle}) implies the existence of a sequence $(x_k)$ such that
    			\begin{equation*}
        			\lim_{k\to \infty} E(x_k) = E^*, ~ \abs{\nabla E} (x_k) < \frac{1}{k}, ~ \Delta_V E (x_k) < \frac{1}{k}.
	    		\end{equation*}
    			Noting $g(V,\nabla E) \geq -\abs{V}\abs{\nabla E} = -\sqrt{2E}\abs{\nabla E}$ and evaluating \cref{V} at $x_k$
    			\begin{align*}
        			\lambda E (x_k) &= \abs{\nabla V}^2 (x_k) - \Delta_V E (x_k) \\
            			&> \abs{\nabla V}^2 (x_k) - \frac{1}{k} 
	    		\end{align*}
    			As $k\to \infty$, we see that $\lambda E^* \geq 0 $, which implies $E=0$ or $\lambda \geq 0$. 
    		\item From \cref{eqn: ricci V V} and the hypothesis, we have $\Delta E = \abs{\nabla V}^2 - \ricci (V,V) \geq 0 $. 
 Consequently, by the hypothesis that  $\abs{\nabla E} \in L^{1}(M)$ and  from \Cref{lem: Subharmonic with finite energy is harmonic},
 we obtain that  $\Delta E = 0$ and hence $\abs{\nabla V}^2 = \ricci(V,V) = 0 $. 
	\end{enumerate}
 \end{proof} 

\begin{example}
Consider    $(\mathbb{R}^2, g, V,\lambda)$   with 
    $$ g=\frac{2}{1+y^2} \brackets{ \derivative x^2 + \derivative y^2 } ,$$
    $V=-x\partial_x -y\partial_y$  and $\lambda=-1$. 
    Its scalar curvature $\scalar = \frac{1-y^2}{1+y^2}$.
    Its energy function 
    $$ E = \frac{x^2 + y^2}{1+y^2} $$
    is not bounded above and also the soliton is expanding.
\end{example}

\begin{example}
    The Sol3 soliton $(\mathbb{R}^3 ,g , V, \lambda ) $ is noncompact, nongradient expanding Ricci soliton with  
    $$ g = e^{2t} \derivative x^2 + e^{-2t} \derivative y^2 + \derivative t^2, $$
    $V= -2x \partial_x +2y\partial_y  $ and $\lambda=-4$. 
    It has constant scalar curvature $\scalar = -2$. 
    See \cite{MR4888461} for a detailed exposition.
    Its energy function 
    $$ E = 2 \left( e^{2t} x^2 + e^{-2t} y^2 \right) $$
    is not bounded above. 
\end{example}


\section{$L^1$-Liouville type theorem with respect to energy function}\label{liou}
    In this section, we prove \Cref{thm: no harmonic function with finite energy,thm: No V-harmonic function with constant energy and scalar bounded below}. To handle the equality case of the lower bound on the scalar curvature in \Cref{thm: no harmonic function with finite energy}, we require a technical lemma.
    We take a brief look at Ricci solitons of constant scalar curvature before proving \Cref{prop: orientable RS constant scalar finite volume mod V notin L1}.

In \cite{MR3415603} gradient Ricci solitons of constant scalar curvature were investigated.  
It is shown that for such solitons the possible values of scalar curvature is of the form 
$\frac{k \lambda}{2}$ for $0\leq k \leq n$. 
The main idea of their proof depends on the theory of isoparametric functions, and analysis of \cref{eqn: S + 2E - lambda f}.
However, an analogue of  \cref{eqn: S + 2E - lambda f} is absent in the nongradient case. 
Nevertheless, we have the following results.
Note that  by \cref{eqn: Divergence of V}  constant $\scalar$ implies that  $\divergence V$ is constant.

\begin{proposition}
    A compact Ricci soliton $(M^n,g,V,\lambda)$  with constant scalar curvature is Einstein. 
\end{proposition}
\begin{proof}
    Follows directly from \cref{eqn: V scalar norm of traceless Ricci} using Stokes' theorem.
\end{proof}

\begin{proposition}\label{Cor: bounds on constant scalar curvature}
     Let $(M^n,g,V,\lambda)$ be a Ricci soliton with $\Delta_V \scalar \geq 0$, then
    \begin{equation}
        0 \leq \scalar \leq \frac{n\lambda}{2} ~\text{ OR }~ \frac{n\lambda}{2} \leq \scalar \leq 0.
    \end{equation}
    Further, 
    \begin{enumerate}
        \item $\lambda=0 $ or $\scalar=0$ implies $M$ is Ricci flat.
        \item $\scalar=\frac{n\lambda}{2}$ implies $M$ is Einstein.
    \end{enumerate}
    In particular, the conclusion holds  for solitons of constant scalar curvature.
    \end{proposition}
\begin{proof}
The proof follows  simply  by employing   $\Delta_V \scalar \geq 0$ in
 \cref{eqn: laplacian Scalar in norm ricci} and \cref{eqn: laplacian Scalar in traceless ricci}.        
In fact, respectively we obtain,   
    \begin{align*}
        \lambda\scalar &= 2\abs{\Qricci}^2 \geq 0 , \\
        \frac{2\scalar}{n} \left( \frac{n\lambda}{2} -\scalar \right) &= 2 \abs{\mathring{\Qricci}}^2 \geq 0.
    \end{align*}
\end{proof}

Gaffney's Stokes theorem states that for a complete (orientable) Riemannian manifold, if $\abs{V} $ and $\divergence V$ are integrable, then $\int (\divergence V) \, \derivative \Vol_g =0$. This implies the following 
\begin{lemma}\label{prop: orientable RS constant scalar finite volume mod V notin L1}
    Let $(M^n,g,V,\lambda)$ be an (orientable) Ricci soliton with constant scalar curvature $\scalar \neq \frac{n\lambda}{2}$ of finite volume, then $\sqrt{E}\not\in \mathcal{L}^1 (M) $.
\end{lemma}
\begin{proof}
As  $M$ has finite volume, constant $\divergence V$ implies $\divergence V$ is integrable. 
If $\sqrt{E}$ is integrable, so is $\abs{V}$, and then Gaffney's Stokes theorem implies $\scalar =  \frac{n\lambda}{2}$, which is contradictory to the given hypothesis.
\end{proof}

    We are now ready to prove \Cref{thm: no harmonic function with finite energy,thm: No V-harmonic function with constant energy and scalar bounded below}.
    Both proofs rely on the divergence theorem to transfer the derivative from $u$ to a cut-off function and then applying the respective Bochner formulas.


\noHarmonicWithFiniteEnergy*
\begin{proof}
    For any point $p$, let $\phi$ be the standard cut-off function on $\mathrm{B}_p(2r)$.
    Let $A_r = \mathrm{B}_p(2r) \setminus \mathrm{B}_p(r) $.
    Let $u$ be a harmonic function satisfying \cref{eqn: finite dirichlet energy}.

    From \cref{eqn: Ricci Soliton equation},
    \begin{equation}\label{eqn: FDE equation one}
         \int_M \phi^2 \ricci (\nabla u, \nabla u)  \, \derivative{\Vol} 
            = \frac{\lambda}{2} \int_M \phi^2  \abs{\nabla u}^2  \, \derivative{\Vol} 
                - \int_M \phi^2 g (\nabla_{\nabla u} V, \nabla u) \, \derivative{\Vol} 
    \end{equation}
    
    Since $\phi$ has compact support, noting $\Delta u = 0$ and applying the divergence theorem to the vector fields $ (Vu) \phi^2 \nabla u $ and $ \phi^2 \abs{\nabla u}^2 V $ yields 
    \begin{align*}
        -\int_M  \phi^2 g(\nabla_{\nabla u} V, \nabla u )  \, \derivative{\Vol}  
            &= \int_M  \phi^2 \hessian u (V,\nabla u)  \, \derivative{\Vol}
                + \int_M (Vu) g( \nabla u, \nabla \phi^2)  \, \derivative{\Vol}   \text{ and } \\
        2\int_M \phi^2 \hessian u (V,\nabla u) \, \derivative{\Vol}
            &= - \int_M  \phi^2 \abs{\nabla u}^2  (\divergence V) \, \derivative{\Vol} 
                - \int_M \abs{\nabla u}^2 g(V,\nabla \phi^2) \, \derivative{\Vol} 
    \end{align*}
    Combined, the above equations give
    \begin{multline}\label{eqn: FDE equation two}
        -\int_M  \phi^2 g(\nabla_{\nabla u} V, \nabla u )  \, \derivative{\Vol}  
            = \int_M (Vu) g( \nabla u, \nabla \phi^2)  \, \derivative{\Vol}  \\
                - \frac{1}{2} \int_M  \phi^2 \abs{\nabla u}^2  (\divergence V) \, \derivative{\Vol} 
                - \frac{1}{2} \int_M \abs{\nabla u}^2 g(V,\nabla \phi^2) \, \derivative{\Vol} 
    \end{multline}

    Now 
    \begin{align*}
        \int_M \phi^2 \Delta \brackets{\frac{1}{2} \abs{\nabla u}^2}  \, \derivative{\Vol}
            &= - \frac{1}{2} \int_M g ( \nabla \phi^2 , \nabla (\abs{\nabla u}^2) ) \, \derivative{\Vol}  \\
            &\leq \frac{1}{2} \int_M 2\phi \abs{\nabla \phi}  2 \abs{\nabla u} \abs{\nabla \abs{\nabla u} }\, \derivative{\Vol} \\
            &\leq \frac{1}{2} \int_M \left( \phi^2 \abs{\nabla \abs{\nabla u} }^2 
                + 4 \abs{\nabla \phi}^2 \abs{\nabla u}^2 \right)  \, \derivative{\Vol}
    \end{align*}
    Using \cref{eqn: Bochner formula for gradient} for harmonic $u$ and Kato's inequality,
    \begin{equation*}
        \Delta \brackets{\frac{1}{2} \abs{\nabla u}^2} 
        = \ricci \brackets{\nabla u, \nabla u} 
         + \abs{ \hessian u}^2 \geq \ricci \brackets{\nabla u, \nabla u} + \abs{ \nabla \abs{\nabla u} }^2. 
    \end{equation*}
    Therefore,
    \begin{equation}\label{eqn: FDE equation three}
        \int_M \phi^2 \ricci (\nabla u, \nabla u) \, \derivative{\Vol} 
            \leq -\frac{1}{2} \int_M  \phi^2 \abs{\nabla \abs{\nabla u} }^2 \, \derivative{\Vol}
                + 2 \int_M \abs{\nabla \phi}^2 \abs{\nabla u}^2  \, \derivative{\Vol}
    \end{equation}

    Combining \cref{eqn: FDE equation one,eqn: FDE equation two,eqn: FDE equation three}
    \begin{multline}
        \int_M \frac{1}{2}  (\lambda - \divergence V) \phi^2 \abs{\nabla u}^2  \, \derivative{\Vol}
        + \frac{1}{2} \int_M  \phi^2 \abs{\nabla \abs{\nabla u} }^2 \, \derivative{\Vol}\\
            \leq 
                 \frac{1}{2} \int_M g( \abs{\nabla u}^2 V-2(Vu) \nabla u , \nabla \phi^2 ) \,\derivative{\Vol}
                + 2 \int_M \abs{\nabla \phi}^2 \abs{\nabla u}^2  \, \derivative{\Vol}
    \end{multline}
    Lastly, noting \cref{eqn: Divergence of V} and  
    $$ \abs{\abs{\nabla u}^2V -2(Vu)  \nabla u }^2 = 2E \abs{\nabla u}^4 ,$$
    we obtain
    \begin{multline}
        \int_M \left( \scalar - (n-2)\frac{\lambda}{2} \right) \phi^2 \abs{\nabla u}^2  \, \derivative{\Vol}
        + \frac{1}{2} \int_M  \phi^2 \abs{\nabla \abs{\nabla u} }^2 \, \derivative{\Vol}\\
            \leq \frac{1}{2} \int_M \sqrt{2E} \abs{\nabla u}^2   \abs{ \nabla \phi^2 } \,\derivative{\Vol}
                + 2 \int_M \abs{\nabla \phi}^2 \abs{\nabla u}^2  \, \derivative{\Vol}
    \end{multline}
    Since $\nabla \phi$ is nonzero only on $A_r$, the integral on the right hand side is over $A_r$. \Cref{eqn: finite dirichlet energy} implies that the right hand side vanishes as $r\to \infty$.

    Since $\scalar \geq \frac{(n-2)}{2} \lambda $, we have that $\abs{\nabla u}$ is constant and at least one of the following is true: $u$ is constant or $\scalar = \frac{(n-2)}{2} \lambda$. 
    Suppose $u$ is not constant and subsequently $\scalar = \frac{(n-2)}{2} \lambda$.
    Constant $\abs{\nabla u}$ and the hypothesis \cref{eqn: finite dirichlet energy} imply that the manifold has finite volume and $\sqrt{E}\in \mathcal{L}^1 (M)$. 
    This is in contradiction to \Cref{prop: orientable RS constant scalar finite volume mod V notin L1}.
\end{proof}

\NoVHarmonic*
\begin{proof}
    For any point $p$, let $\phi$ be the standard cut-off function on $\mathrm{B}_p(2r)$.
    Let $A_r = \mathrm{B}_p(2r) \setminus \mathrm{B}_p(r) $.
    Throughout this proof, we will denote $\mathcal{U} = \frac{1}{2} \abs{\nabla u}^2$.
    \begin{align*}
        - \int_M \brackets{V \mathcal{U} } \phi^2 \, \derivative {\Vol}_g
            &= \int_M \mathcal{U} \left( \phi^2 \divergence V 
                + 2\phi V\phi \right) \, \derivative {\Vol}_g \\
            &\leq \int_M \mathcal{U} (\phi^2 \divergence V) \, \derivative {\Vol}_g
                + \int_{A_r} \mathcal{U} (2\phi \abs{V}\abs{\nabla \phi}) 
                    \, \derivative {\Vol}_g \\
            &\leq \int_M \mathcal{U} (\phi^2 \divergence V) \, \derivative {\Vol}_g
                + \int_{A_r} \mathcal{U} 
                    \left( \phi^2 \abs{V}^2  + \abs{\nabla \phi}^2  \right) 
                    \, \derivative {\Vol}_g \\
            &\leq \int_M \mathcal{U} (\phi^2 \divergence V) \, \derivative {\Vol}_g
                + \int_{A_r} \mathcal{U} 
                    \left( 2\phi^2 E + \abs{\nabla \phi}^2  \right) 
                    \, \derivative  {\Vol}_g .
    \end{align*}
    Also,
    \begin{align*}
        \int_M \brackets{ \Delta\mathcal{U} } \phi^2 \, \derivative {\Vol}_g 
            &= - \int_M  2\phi g \brackets{ \nabla\mathcal{U} , \nabla \phi  }
                \, \derivative {\Vol}_g \\
            &= - \int_{A_r} 2\phi \hessian u ( \nabla u, \nabla\phi  )
                \, \derivative {\Vol}_g \\
            &\leq \int_{A_r} 2\phi \abs{\hessian u} \abs{\nabla u} \abs{\nabla \phi}
                \, \derivative {\Vol}_g \\
            &\leq \int_{A_r} \brackets{ \phi^2 \abs{\hessian u}^2 
                + \abs{\nabla u}^2 \abs{\nabla \phi}^2 }
                \, \derivative {\Vol}_g \\
            &\leq \int_{A_r} \brackets{ \phi^2 \abs{\hessian u}^2 
                + 2\mathcal{U} \abs{\nabla \phi}^2 }
                \, \derivative {\Vol}_g .
    \end{align*}
    Using the above inequalities, 
    \begin{multline*}
        \int_M  \brackets{ \Delta_V \mathcal{U}} \phi^2 \, \derivative {\Vol}_g 
            \leq \int_M \mathcal{U}  \left( \phi^2 \divergence V  \right) 
                \, \derivative  {\Vol}_g 
                + \int_{A_r} \brackets{ \phi^2 \abs{\hessian u}^2  }
                \, \derivative {\Vol}_g \\
            +\int_{A_r} \brackets{  3\mathcal{U} \abs{\nabla \phi}^2}
                \, \derivative {\Vol}_g 
            + \int_{A_r} (2\phi^2 \mathcal{U} E) \, \derivative {\Vol}_g.
    \end{multline*}
    Comparing this with \cref{eqn: Bochner formula w.r.t. Delta V} for a $V$-harmonic function $u$
    and noting, 
    $$ \int_M \abs{\hessian u}^2 \phi^2 \, \derivative {\Vol}_g \geq 
        \int_{A_r} \abs{\hessian u}^2 \phi^2 \, \derivative  {\Vol}_g, $$
    we obtain
    \begin{equation*}
        \int_M \brackets{ \lambda - \divergence V } \mathcal{U} \phi^2
                 \, \derivative  {\Vol}_g 
            \leq \int_{A_r} \brackets{  3\mathcal{U} \abs{\nabla \phi}^2}
                \, \derivative {\Vol}_g 
            + \int_{A_r} (2\phi^2 \mathcal{U} E) \, \derivative {\Vol}_g.
    \end{equation*}
    Lastly, taking $r\to \infty$, and using \cref{eqn: Divergence of V},
    \begin{equation*}
        0 \leq \int_M  \brackets{ \scalar - \frac{n-2}{2} \lambda } \brackets{ \frac{1}{2} \abs{\nabla u}^2 }
                  \, \derivative {\Vol}_g
            \leq \lim_{r\to \infty} \int_{A_r} \brackets{ \frac{1}{2} \abs{\nabla u}^2 } \left(2 E \right) 
                \, \derivative {\Vol}_g =0.
    \end{equation*}
    By hypothesis on scalar curvature, the above inequality implies that $u$ is constant. 
\end{proof}




\begin{example}[p. 79,  \cite{MR4624811}]\label{ex: hamilton cigar soliton}
    Hamilton's Cigar soliton, $(\mathbb{R}^2 , g, \nabla f, 0)$ is a noncompact steady gradient Ricci soliton where 
    $$ g = \frac{ 4 \left( \derivative x^2 + \derivative y^2 \right) }{1+x^2+y^2} , $$
    and $f= - \ln \brackets{ 1+ x^2 + y^2 }$. 
    Its scalar curvature $\scalar = \frac{1}{ 1+  x^2 +  y^2 } > 0$.
    The gradient of $f$ is 
    $$ \nabla f = - \frac{1}{2} (x \partial_x + y \partial_y). $$ 
    Since, the energy function 
    $$ E= \frac{1}{2} \frac{x^2 + y^2 }{1+ x^2 + y^2} \leq \frac{1}{2} $$
    is bounded above, \Cref{thm: No V-harmonic function with constant energy and scalar bounded below,thm: no harmonic function with finite energy} imply
    that there is no $f$-harmonic or harmonic function $u$ on $(\mathbb{R}^2,g)$ such that $\abs{\nabla u}^2  \in L^{1}(M)$. 
    Thus, the manifold has at most one nonparabolic end.
\end{example}

\begin{remark}
Note that this can be generalised to the broader Bryant soliton case which is also a steady Ricci soliton.
Since the scalar curvature of Bryant soliton is  positive,  and  by our previous discussion 
$\scalar + 2E$ is constant (see \cref{eqn: S + 2E - lambda f}). Consequently, $E$ is bounded above. 
Hence,  the same conclusion as above holds.
\end{remark}

\section{Finiteness of $E$-Volume of Shrinking Ricci Solitons}\label{finite-volume}

We now turn our attention to proving \Cref{thm: bounded Ricci implies finite E volume}.
We first begin with a lower estimate on the energy function of a Ricci soliton with bounded Ricci curvature.
\begin{lemma}
	Let $(M^n,g,V,\lambda)$ be a noncompact shrinking Ricci soliton with the bounded Ricci curvature. 
Then  there exists an $L_0$ such that for every 
	unit speed geodesic $\gamma: [0,L] \to M $ with $L\geq L_0$ and 
	$\gamma (0) = p \in M $,
	$$ E (L) \geq \frac{\lambda^2}{8} L^2 + \frac{\lambda}{2} a L + \frac{a^2}{2},  $$
	where $a$ is a constant that depends only on the geometry of the soliton on the unit ball $B(p,1)$.
\end{lemma}
\begin{proof}
	We first assume that $L\geq 2$.
	Let $E^{i}(p)$ denote an orthonormal basis along $\gamma$ with $E^n(t) = {\gamma}'(t)$ and 
   $E^{i}(t)$  denote  the parallel transport of $E^i$ over $\gamma(t)$.  Write
$V = V^i E_i$. Then for any Lipschitz function $\phi:[0,L] \to [0,1]$ with $\phi(0) = \phi(L)=0$, the second variation formula implies
	$$ (n-1) \int_0^L (\phi ' )^2 \, \derivative t 
		\geq \int_0^L \phi^2 \ricci (\gamma' , \gamma ' ) \, \derivative t. $$
	 Using the Ricci soliton equation and 
	$ \frac{1}{2} \lie{V} g (\gamma' , \gamma ' ) = \dot{V^n}, $
	in the above equation, we get
	$$ (n-1) \int_0^L (\phi ' )^2 \, \derivative t 
		\geq \frac{\lambda}{2} \int_0^L \phi^2 \, \derivative t 
		- \int_0^L \phi^2 \dot{V^n} \, \derivative t  . $$
	Subsequently,
	$$ (n-1) \int_0^L (\phi ' )^2 \, \derivative t 
		\geq \frac{\lambda}{2} \int_0^L \phi^2 \, \derivative t 
		- \int_0^L \dot{V^n} \, \derivative t		
		+ \int_0^L (1-\phi^2) \dot{V^n} \, \derivative t . $$
Using  the Ricci soliton equation again and noting the bound on Ricci curvature,
	\begin{align*}
		\int_0^L \left(1-\phi^2\right) \dot{V^n} \, \derivative t &= 
		\int_0^L \left(1-\phi^2\right) \left( \frac{\lambda}{2} - \ricci \left(\gamma' , \gamma'\right) \right) \, \derivative t \\
		&\geq \int_0^L \left(1-\phi^2\right) \left( \frac{\lambda}{2} - C \right) \, \derivative t. 
	\end{align*}
	Therefore,
	\begin{equation}\label{eqn: Lower estimate on Vn}
		(n-1) \int_0^L (\phi ' )^2 \, \derivative t 
		\geq - V^n(L) + V^n(0)		
		+ \int_0^L \left( \frac{\lambda}{2} - C \right)  \, \derivative t  
		+ C \int_0^L \phi^2 \, \derivative t .
	\end{equation}
	Now we choose,
	\begin{equation*}
		\phi = 
			\begin{cases}
				t, 	& 0 \leq t \leq 1 ; \\
				1, 	& 1 \leq t \leq L-1 ; \\
				L-t, 	& L-1 \leq t \leq L. \\
			\end{cases} 
	\end{equation*}
	Substituting in \cref{eqn: Lower estimate on Vn},
	$$ V^n (L) \geq V^n (0) - 2 (n-1) +\frac{\lambda}{2} L - \frac{4}{3} C , $$
	or
	$$ \sqrt{2E} (L) \geq V^n(L) \geq \frac{\lambda}{2} L + a, $$
	where $a$ is a constant that depends only on the geometry of the soliton on the unit ball $B(p,1)$. Since $\lambda>0$, we can choose $L_0$ large enough so that the right hand side of the above equation is positive. Then  we have for $L\geq L_0$, 
	$$ E (L) \geq \frac{\lambda^2}{8}L^2 + \frac{\lambda}{2} a L + \frac{a^2}{2}. $$
\end{proof}

The proof of \Cref{thm: bounded Ricci implies finite E volume} employs the same technique as in \cite{naber2006geometryanalysisriccisolitons}. 
The only major point of difference is the usage of estimates for $E$ in place of $f$. 
\shrinkingRShasfiniteEVolume*
\begin{proof}
	Using exponential coordinates at $p$ we have, since $\ricci\geq -C$, by the standard comparison that 
	$\sqrt{ \mathrm{det} g} \leq (\sinh \sqrt{C} r)^{n-1} \lesssim e^{(n-1)\sqrt{C} r} . $ Here  notation, $s\lesssim t$ means $s \leq At$ for some constant $A$. 
	\begin{multline*}
		{\Vol}_E (M) = \int_{S^{n-1}} \int_0^\infty e^{-E} \sqrt{\mathrm{det} g} \, \derivative r \, \derivative s^{n-1} \\ 
		\lesssim \int_0^{L_0} e^{-E} e^{(n-1)\sqrt{C} r} \derivative r + 
		\int_{L_0}^\infty e^{ -\frac{\lambda^2}{8}r^2 + \left( (n-1)\sqrt{C} -\frac{\lambda}{2} a \right) r - \frac{a^2}{2} } \, \derivative r < \infty .
	\end{multline*}
 It should be observed  that  $\sqrt{ \mathrm{det} g}(x) = 0$ for $x$ outside the segment domain of $p$.  	
Hence, all of the above integrals  are meaningful.

Now we lift to the universal cover $ (\tilde{M}, \tilde{g} ) $ of $(M,g)$ and note that $ ( \tilde{M}, \tilde{g} , \tilde{V}, \lambda )$ is a shrinking Ricci soliton, where $\tilde{V}$ is the pullback of $V$.
Furthermore, the corresponding energy function $\tilde{E}$ is also the pullback of the energy function $E$ of $M$. 
Since both the integrals in
\begin{equation*}
	\int_{\tilde{M}} e^{-\tilde{E}} \, \derivative {\Vol}_{\tilde{g}}  = 
		\abs{ \pi_1(M) } \int_M e^{-E} \, \derivative {\Vol}_g 
\end{equation*}
are finite, it follows that $ \pi_1(M)$ is finite.
\end{proof}
We end this article by recovering the result of \cite{naber2006geometryanalysisriccisolitons} for gradient shrinking Ricci solitons using the \Cref{thm: bounded Ricci implies finite E volume}.

\begin{corollary}\label{cor: shrinking GRS has finite f volume}
	A complete gradient shrinking Ricci soliton $(M^n,g, f,\lambda)$ with bounded Ricci curvature has finite $f$-volume.
\end{corollary}
\begin{proof}
	Without loss of generality, we can assume that $\lambda = 2$ and that \cref{eqn: S + 2E - lambda f} gives
	$$ \scalar + \abs{\nabla f}^2 - 2 f = 0 .$$
	Also, it is known for shrinking gradient Ricci solitons that $\scalar \geq 0$ (see Theorem 2.14, \cite{MR4624811}). Combining these imply 
	$$ E = \frac{1}{2} \abs{ \nabla f }^2 \leq f .$$
	\Cref{thm: bounded Ricci implies finite E volume} then implies
	$$ \int_M e^{-f} \,\derivative {\Vol}_g \leq \int_M e^{-E} \,\derivative {\Vol}_g < \infty . $$
\end{proof}

\printbibliography

@misc{perelman2002entropyformularicciflow,
      title={The entropy formula for the Ricci flow and its geometric applications}, 
      author={Grisha Perelman},
      year={2002},
      eprint={math/0211159},
      archivePrefix={arXiv},
      primaryClass={math.DG},
      url={https://arxiv.org/abs/math/0211159}, 
}

@article {MR2507581,
    AUTHOR = {Petersen, Peter and Wylie, William},
     TITLE = {Rigidity of gradient {R}icci solitons},
   JOURNAL = {Pacific J. Math.},
  FJOURNAL = {Pacific Journal of Mathematics},
    VOLUME = {241},
      YEAR = {2009},
    NUMBER = {2},
     PAGES = {329--345},
      ISSN = {0030-8730,1945-5844},
   MRCLASS = {53C24 (53C21 53C25)},
  MRNUMBER = {2507581},
MRREVIEWER = {Esther\ Cabezas Rivas},
       DOI = {10.2140/pjm.2009.241.329},
       URL = {https://doi.org/10.2140/pjm.2009.241.329},
}

@incollection {MR4518934,
    AUTHOR = {Yau, Shing-Tung},
     TITLE = {Some function-theoretic properties of complete {R}iemannian
              manifold and their applications to geometry},
 BOOKTITLE = {Selected works of {S}hing-{T}ung {Y}au. {P}art 1. 1971--1991.
              {V}ol. 2. {M}etric geometry and harmonic functions},
     PAGES = {173--184},
      NOTE = {Reprint of [0417452]},
 PUBLISHER = {Int. Press, Boston, MA},
      YEAR = {2019},
      ISBN = {978-1-57146-377-7},
   MRCLASS = {},
  MRNUMBER = {4518934},
}

@book {MR284950,
    AUTHOR = {Yano, Kentaro},
     TITLE = {Integral formulas in {R}iemannian geometry},
    SERIES = {Pure and Applied Mathematics},
    VOLUME = {No. 1},
 PUBLISHER = {Marcel Dekker, Inc., New York},
      YEAR = {1970},
     PAGES = {ix+156},
   MRCLASS = {53.72},
  MRNUMBER = {284950},
MRREVIEWER = {T.\ Nagano},
}

@incollection {MR2648937,
    AUTHOR = {Cao, Huai-Dong},
     TITLE = {Recent progress on {R}icci solitons},
 BOOKTITLE = {Recent advances in geometric analysis},
    SERIES = {Adv. Lect. Math. (ALM)},
    VOLUME = {11},
     PAGES = {1--38},
 PUBLISHER = {Int. Press, Somerville, MA},
      YEAR = {2010},
      ISBN = {978-1-57146-143-8},
   MRCLASS = {53C21 (53C25 53C44)},
  MRNUMBER = {2648937},
MRREVIEWER = {Manuel\ Fern\'andez-L\'opez},
}

@book {MR4624811,
    AUTHOR = {Chow, Bennett},
     TITLE = {Ricci solitons in low dimensions},
    SERIES = {Graduate Studies in Mathematics},
    VOLUME = {235},
 PUBLISHER = {American Mathematical Society, Providence, RI},
      YEAR = {2023},
     PAGES = {xvi+339},
      ISBN = {9781470474287},
   MRCLASS = {53E20 (30F20 53-02 53C25 58-02)},
  MRNUMBER = {4624811},
MRREVIEWER = {Xiaolong\ Li},
       DOI = {10.1090/gsm/235},
       URL = {https://doi.org/10.1090/gsm/235},
}

@article {MR2448435,
    AUTHOR = {Eminenti, Manolo and La Nave, Gabriele and Mantegazza, Carlo},
     TITLE = {Ricci solitons: the equation point of view},
   JOURNAL = {Manuscripta Math.},
  FJOURNAL = {Manuscripta Mathematica},
    VOLUME = {127},
      YEAR = {2008},
    NUMBER = {3},
     PAGES = {345--367},
      ISSN = {0025-2611,1432-1785},
   MRCLASS = {53C21 (53C44)},
  MRNUMBER = {2448435},
MRREVIEWER = {Esther\ Cabezas Rivas},
       DOI = {10.1007/s00229-008-0210-y},
       URL = {https://doi.org/10.1007/s00229-008-0210-y},
}

@book {MR3469435,
    AUTHOR = {Petersen, Peter},
     TITLE = {Riemannian geometry},
    SERIES = {Graduate Texts in Mathematics},
    VOLUME = {171},
   EDITION = {Third},
 PUBLISHER = {Springer, Cham},
      YEAR = {2016},
     PAGES = {xviii+499},
      ISBN = {978-3-319-26652-7},
   MRCLASS = {53-01 (53C20 53C21 53C23)},
  MRNUMBER = {3469435},
       DOI = {10.1007/978-3-319-26654-1},
       URL = {https://doi.org/10.1007/978-3-319-26654-1},
}

@article {MR4393949,
    AUTHOR = {Ghosh, Amalendu},
     TITLE = {On {B}ach almost solitons},
   JOURNAL = {Beitr. Algebra Geom.},
  FJOURNAL = {Beitr\"age zur Algebra und Geometrie. Contributions to Algebra
              and Geometry},
    VOLUME = {63},
      YEAR = {2022},
    NUMBER = {1},
     PAGES = {45--54},
      ISSN = {0138-4821,2191-0383},
   MRCLASS = {53C25 (53C15 53C20)},
  MRNUMBER = {4393949},
       DOI = {10.1007/s13366-021-00565-4},
       URL = {https://doi.org/10.1007/s13366-021-00565-4},
}

@article {sym12020289,
AUTHOR = {Deshmukh, Sharief and Alsodais, Hana},
TITLE = {A Note on Ricci Solitons},
JOURNAL = {Symmetry},
VOLUME = {12},
YEAR = {2020},
NUMBER = {2},
ARTICLE-NUMBER = {289},
URL = {https://www.mdpi.com/2073-8994/12/2/289},
ISSN = {2073-8994},
DOI = {10.3390/sym12020289}
}

@article {MR3415603,
    AUTHOR = {Fern\'andez-L\'opez, Manuel and Garc\'ia-R\'io, Eduardo},
     TITLE = {On gradient {R}icci solitons with constant scalar curvature},
   JOURNAL = {Proc. Amer. Math. Soc.},
  FJOURNAL = {Proceedings of the American Mathematical Society},
    VOLUME = {144},
      YEAR = {2016},
    NUMBER = {1},
     PAGES = {369--378},
      ISSN = {0002-9939,1088-6826},
   MRCLASS = {53C25 (53C20 53C44)},
  MRNUMBER = {3415603},
MRREVIEWER = {Ramiro\ Augusto\ Lafuente},
       DOI = {10.1090/proc/12693},
       URL = {https://doi.org/10.1090/proc/12693},
}

@article {MR4888461,
    AUTHOR = {Belarbi, Lakehal},
     TITLE = {Ricci solitons of the {$Sol_3$} {L}ie group},
   JOURNAL = {J. Indian Math. Soc. (N.S.)},
  FJOURNAL = {The Journal of the Indian Mathematical Society. New Series},
    VOLUME = {92},
      YEAR = {2025},
    NUMBER = {2},
     PAGES = {329--339},
      ISSN = {0019-5839,2455-6475},
   MRCLASS = {53C25 (53C30 53C50)},
  MRNUMBER = {4888461},
MRREVIEWER = {Hui\ Zhang},
}

@misc{naber2006geometryanalysisriccisolitons,
      title={Some Geometry and Analysis on Ricci Solitons}, 
      author={Aaron Naber},
      year={2006},
      eprint={math/0612532},
      archivePrefix={arXiv},
      primaryClass={math.DG},
      url={https://arxiv.org/abs/math/0612532}, 
}

@article {MR3546774,
    AUTHOR = {Catino, G. and Mastrolia, P. and Monticelli, D. D. and Rigoli,
              M.},
     TITLE = {Analytic and geometric properties of generic {R}icci solitons},
   JOURNAL = {Trans. Amer. Math. Soc.},
  FJOURNAL = {Transactions of the American Mathematical Society},
    VOLUME = {368},
      YEAR = {2016},
    NUMBER = {11},
     PAGES = {7533--7549},
      ISSN = {0002-9947,1088-6850},
   MRCLASS = {53C20 (53C25)},
  MRNUMBER = {3546774},
MRREVIEWER = {Jeffrey\ Jauregui},
       DOI = {10.1090/tran/6864},
       URL = {https://doi.org/10.1090/tran/6864},
}

@article {MR3063556,
    AUTHOR = {Mastrolia, Paolo and Rigoli, Marco and Rimoldi, Michele},
     TITLE = {Some geometric analysis on generic {R}icci solitons},
   JOURNAL = {Commun. Contemp. Math.},
  FJOURNAL = {Communications in Contemporary Mathematics},
    VOLUME = {15},
      YEAR = {2013},
    NUMBER = {3},
     PAGES = {1250058, 25},
      ISSN = {0219-1997,1793-6683},
   MRCLASS = {53C21},
  MRNUMBER = {3063556},
MRREVIEWER = {Uday\ Chand\ De},
       DOI = {10.1142/S0219199712500587},
       URL = {https://doi.org/10.1142/S0219199712500587},
}

@article {MR3721584,
    AUTHOR = {Sharma, Ramesh and Balasubramanian, S. and Uday Kiran, N.},
     TITLE = {Some remarks on {R}icci solitons},
   JOURNAL = {J. Geom.},
  FJOURNAL = {Journal of Geometry},
    VOLUME = {108},
      YEAR = {2017},
    NUMBER = {3},
     PAGES = {1031--1037},
      ISSN = {0047-2468,1420-8997},
   MRCLASS = {53C25 (53C44 53C55)},
  MRNUMBER = {3721584},
MRREVIEWER = {Adara\ M.\ Blaga},
       DOI = {10.1007/s00022-017-0392-0},
       URL = {https://doi.org/10.1007/s00022-017-0392-0},
}

@article {MR2373611,
    AUTHOR = {Wylie, William},
     TITLE = {Complete shrinking {R}icci solitons have finite fundamental
              group},
   JOURNAL = {Proc. Amer. Math. Soc.},
  FJOURNAL = {Proceedings of the American Mathematical Society},
    VOLUME = {136},
      YEAR = {2008},
    NUMBER = {5},
     PAGES = {1803--1806},
      ISSN = {0002-9939,1088-6826},
   MRCLASS = {53C20},
  MRNUMBER = {2373611},
MRREVIEWER = {Eduardo\ Garc\'ia-R\'io},
       DOI = {10.1090/S0002-9939-07-09174-5},
       URL = {https://doi.org/10.1090/S0002-9939-07-09174-5},
}

@article {MR1158340,
    AUTHOR = {Li, Peter and Tam, Luen-Fai},
     TITLE = {Harmonic functions and the structure of complete manifolds},
   JOURNAL = {J. Differential Geom.},
  FJOURNAL = {Journal of Differential Geometry},
    VOLUME = {35},
      YEAR = {1992},
    NUMBER = {2},
     PAGES = {359--383},
      ISSN = {0022-040X,1945-743X},
   MRCLASS = {53C21 (53C20 58G30)},
  MRNUMBER = {1158340},
MRREVIEWER = {Yang\ Lian\ Pan},
       URL = {http://projecteuclid.org/euclid.jdg/1214448079},
}

@article {MR3023848,
    AUTHOR = {Munteanu, Ovidiu and Sesum, Natasa},
     TITLE = {On gradient {R}icci solitons},
   JOURNAL = {J. Geom. Anal.},
  FJOURNAL = {Journal of Geometric Analysis},
    VOLUME = {23},
      YEAR = {2013},
    NUMBER = {2},
     PAGES = {539--561},
      ISSN = {1050-6926,1559-002X},
   MRCLASS = {53C44},
  MRNUMBER = {3023848},
MRREVIEWER = {Xiang\ Gao},
       DOI = {10.1007/s12220-011-9252-6},
       URL = {https://doi.org/10.1007/s12220-011-9252-6},
}

@article {MR2843238,
    AUTHOR = {Munteanu, Ovidiu and Wang, Jiaping},
     TITLE = {Smooth metric measure spaces with non-negative curvature},
   JOURNAL = {Comm. Anal. Geom.},
  FJOURNAL = {Communications in Analysis and Geometry},
    VOLUME = {19},
      YEAR = {2011},
    NUMBER = {3},
     PAGES = {451--486},
      ISSN = {1019-8385,1944-9992},
   MRCLASS = {53C23 (53C21 53C44)},
  MRNUMBER = {2843238},
MRREVIEWER = {Emil\ Saucan},
       DOI = {10.4310/CAG.2011.v19.n3.a1},
       URL = {https://doi.org/10.4310/CAG.2011.v19.n3.a1},
}

\vspace{1in}

\noindent
{\bf Acknowledgement:} The first author thanks Harish-Chandra Research Institute and Homi Bhabha National Institute (HBNI) for its research fellowship.

\section{Declarations}
	
	\begin{itemize}
		\item Funding: Not applicable. 
		
		\item Conflict of interest/Competing interests: The authors have no conflict of interest and no financial interests in this article.
		
		\item Ethics approval: The submitted work is original and not submitted to more than one journal for simultaneous consideration.
		
		\item Consent to participate: Not applicable.
		
		\item Consent for publication: Not applicable.
		
		\item Availability of data and materials: This manuscript has no associated data.
		
		\item Code availability: Not applicable.
		
		\item Authors' contributions: Conceptualisation, methodology, investigation, validation, writing-original draft, review, editing, and heading have been performed by all the paper's authors.
		
	\end{itemize}

\end{document}